\documentclass[10pt,a4paper]{article}

\usepackage{latexsym}
\usepackage{verbatim}
\usepackage{amsthm}
\usepackage{amsfonts}
\usepackage{amsmath}
\usepackage{amssymb}
\usepackage{amscd}

\numberwithin{equation}{section}

\newtheorem{prop}{Proposition}[section]
\newtheorem{theorem}[prop]{Theorem}

\newtheorem{definition}[prop]{Definition}

\def\begeq{\begin{equation}}
\def\endeq{\end{equation}}

\def\<{\langle}
\def\>{\rangle}
\def\({\left(}
\def\){\right)}

\def\p{\partial}
\def\si{\sigma}
\def\om{\omega}
\def\Om{\Omega}
\def\we{\wedge}
\def\al{\alpha}
\def\be{\beta}
\def\ep{\epsilon}

\DeclareMathOperator{\sh}{\mathcal {SH}}

 \DeclareMathOperator{\la}{\lambda}

\begin{document}

\title{ Liouville and Calabi-Yau type theorems for complex
Hessian equations}
\author{S\l awomir Dinew, S\l awomir Ko\l odziej}
\date{}
\maketitle
\begin{abstract} We prove a Liouville type theorem for entire maximal $m$-subharmonic functions
in $\mathbb C^n$ with bounded gradient. This result, coupled with
 a standard blow-up argument, yields a (non-explicit) a
priori gradient estimate for the complex Hessian equation on a
compact K\"ahler manifold. This terminates the program, initiated
in \cite{HMW},  of solving the non-degenerate Hessian equation on
such manifolds in full generality. We also obtain, using our
previous work, continuous weak solutions in the degenerate case
for the right hand side in some $L^p , $ with a sharp bound on
$p$.

\end{abstract}
\section*{Introduction}
The complex Hessian equation $$ (\om+dd^cu)^{m}\we\om^{n-m}=f\om^n
$$
($\om$ a K\"ahler form, $f>0$ given function, $1<m<n$) can be seen
as an intermediate step between the Laplace and the complex
Monge-Amp\`ere equation. It encompasses the most natural
invariants of the complex Hessian matrix of a real valued function
- the elementary symmetric polynomials of its eigenvalues. Its
real counterpart has numerous applications and has been thoroughly
studied (see \cite{W}). Some equations of similar type appear in
the study of geometric deformation flows such as the $J$- flow
(\cite{Ch2,Ch3,D, SW}). Thus it can be expected that Hessian equations
themselves may also have interesting applications in
 geometry, and indeed they came up recently  in \cite{AV} in  problems related to {\it
 quaternionic geometry}.

Due to their non linear structure, and similarly to the
Monge-Amp\`ere equation, the Hessian equation
 is considered only over a suitable subclass of functions for which ellipticity is guaranteed.
 These subclasses, in a sense, interpolate between subharmonic and plurisubharmonic functions and
 thus they are also interesting  to study from potential theoretic point of view (\cite{Bl},
 \cite{DK}, \cite{Chi}). In the K\"ahler manifold case the corresponding classes of functions,
 namely $\om-m$ subharmonic ones, do not yield, in general, classical positive definite
 metrics, which distinguishes them from  $\om$-plurisubharmonic functions.

 The {\it real} Hessian
has been extensively
studied, both in suitable domains of $\mathbb R^n$, and on
Riemannian manifolds. Likewise, the theory of the complex Hessian
equation is being developed in domains of $\mathbb C^n$ and on
compact K\"ahler manifolds. The case of domains in $\mathbb C^n$
is to a large extent understood (\cite{Li}, \cite{Bl}, \cite{DK}).
In particular the classical Dirichlet problem is solvable in the
smooth category (\cite{Li}, \cite{Bl}) under  standard regularity
assumptions on the data and some convexity assumptions on the
boundary of the  domain.

The corresponding  problem on a compact $n$-dimensional K\"ahler
manifold $(X,\om)$
\begin{equation}\label{AA}
\begin{cases}u\ {\rm is}\ \om-m\ {\rm subharmonic}\\
 (\om+dd^cu)^{m}\we\om^{n-m}=f\om^n\\
sup_Xu=0,
\end{cases}
\end{equation}
with $f$ smooth, strictly positive, and satisfying the necessary
assumption $\int_Xf\om^n=\int_X\om^n$ remained unsolved in full
generality until now.  The equation was solved  under the
assumption of nonnegative holomorphic bisectional curvature
(\cite{Hou}, \cite{J}, \cite{Kok}) or under additional convexity
assumptions on $u$ itself (\cite{G}, \cite{Z}).

The natural approach by continuity method, as in the case of the
complex Monge-Amp\`ere equation (\cite{Y}), faces an obstacle
which does not appear in Yau's solution to the Calabi conjecture:
the $\mathcal C^2$ a priori estimate needed for the closedness
part seems to depend on the a priori gradient bound. Such a
gradient bound in the Monge-Amp\`ere case was obtained by B\l ocki
and Guan independently (\cite{Bl2}, \cite{G}) but their methods
seem to fail in the Hessian case.

Thus the only remaining piece in the program of solving  the
complex Hessian equation on compact K\"ahler manifolds is the a
priori gradient estimate.

In this note we prove such an estimate by exploiting blow-up
analysis. This approach was advocated by Hou-Ma-Wu in \cite{HMW}. In the setting of complex fully nonlinear  equations the blow-up analysis was introduced by Chen in \cite{Ch1}. In essence it boils down to proving a Liouville type theorem for a suitable class of entire functions in $\mathbb C^n$. The essential difficulty that arises is that just like in the case of entire subharmonic functions in $\mathbb R^n,\ n\geq 3$, there are  bounded, nonconstant and entire $m$-subharmonic functions. Nevertheless we obtain a
Liouville type theorem for entire  {\it maximal} $m$-subharmonic
functions in $\mathbb C^n$ and the gradient bound then follows as
a corollary. Below we give the precise statements:
\begin{theorem}[\bf Main result]
Let $u$ be a globally bounded maximal $m$ subharmonic function in $\mathbb C^n,\ 1<m<n$.
Suppose moreover that $sup_{\mathbb C^n}||\nabla u||<\infty$. Then $u$ is constant.
\end{theorem}
\begin{theorem}[\bf Gradient estimate] Let $(X,\om)$ be a compact K\"ahler manifold of complex
dimension $n$. Suppose that $u\in\mathcal C^4(X)$ is an $\om-m$
subharmonic function solving the equation
$$
 (\om + dd^c u)^m \we \om ^{n-m} = f\om ^n
$$
and $\max u =0$, with $f$ a $\mathcal C^2$ smooth positive function. Then
$$sup_X||\nabla u||_{\om}\leq C,$$
with a constant $C$ dependent on $n,\ m,\ ||f^{1/m}||_{\mathcal C^2}$ and the bound on the
bisectional curvature of $(X,\om)$.
\end{theorem}

Let us observe that the estimate remains valid for nonnegative $f$ provided $||f^{1/m}||_{\mathcal C^2}$ is bounded.

As  a corollary, our bound, coupled with the existent results
(\cite{HMW}, see  Preliminaries), yields that the equation
(\ref{AA}) is solvable in the smooth category.

\begin{theorem}\label{HE}
The  equation (\ref{AA}) has a solution on any compact K\"ahler
manifold.
\end{theorem}

From Theorem \ref{HE} and the results of our previous paper
\cite{DK} we also obtain weak continuous solutions for nonnegative
$f$ in some $L^p$ spaces. The result is sharp as far as the
exponent $p$ is concerned.

\begin{theorem}\label{weaksolK} Let $X$ be a compact K\"ahler manifold.
 Then for $p>n/m $ and $f\in L^p (dV)$  there exists a unique
function $u\in \mathcal{SH} _m (X, \om )$ satisfying
$$
 (\om + dd^c u)^m \we \om ^{n-m} = f\om ^n
$$
and $\max u =0.$
\end{theorem}
In \cite{DK} this theorem was shown under additional assumption of
nonnegative holomorphic bisectional curvature. The same proof
applies here once we have smooth solutions from Theorem \ref{HE}.

\bigskip
\section{Preliminaries}
We briefly recall the notions and results that we shall need later on.

Throughout the note $\Om$ will denote a relatively compact domain
in some $\mathbb C^n$, while $(X,\om)$ will be a compact K\"ahler
manifold of complex dimension $n$. Let also
 $d=\partial+\bar{\partial}$ and $d^c:=i(\bar{\partial}-\partial)$ denote the
 standard exterior differentiation operators. By $\beta:=dd^c||z||^2$ we shall
 denote
 the Euclidean K\"ahler form in $\mathbb C^n$, while $\om$ will be a fixed
 K\"ahler form on a compact K\"ahler manifold.

The standard {\it positivity cones} associated to $\si_m$ equations are defined by
\begin{equation}\label{ga}
\Gamma_m=\lbrace \la\in\mathbb R^n|\si_1(\la)> 0,\ \cdots,\
\si_m(\la)> 0\rbrace,
\end{equation}
with $\si_j(\la):=\sum _{1\leq i_1 < ... < i_j \leq n }\la _{i_1}\la _{i_2}
... \la _{i_m}$.

By G\aa rding theorem (see \cite{Ga})
these positivity cones are  convex.

If $\Om$ is a fixed domain in $\mathbb C$ and $u$ is any $\mathcal
C^2(\Om)$ function we call it $m-\beta$-subharmonic ($m$-sh for
short) if for any $z\in\Om$ the complex Hessian matrix
$\frac{\partial^2u} {\partial z_i\partial\bar{z}_j}(z)$ has
eigenvalues forming a vector in the closure of the cone
$\Gamma_m$. By analogy given a complex K\"ahler manifold $(X,\om)$
a $\mathcal C^2$ smooth function $u$ is called $\om-m$-subharmonic
if the form $\om+dd^cu$ has at each point $z$ eigenvalues (with
respect to $\om$) forming a vector in $\bar{\Gamma }_m$. Note that
as $\om$ depends on $z$ the positivity assumption is {\it
dependent} on $z$ too (it can be proved (\cite{Hou}) that
the positivity cones associated to any point $z$ are invariant
under parallel transport with respect to the Levi-Civita
connection of $\om$.).

In the language of differential forms $u$ is $m$-subharmonic (respectively
$\om-m$ subharmonic)
if and only if the following inequalities hold:
$$(dd^cu)^k\we\om^{n-k}\geq 0,\ k=1,\ \cdots,\ m,$$
(resp. $(\om+dd^cu)^k\we\om^{n-k}\geq 0,\ k=1,\ \cdots,\ m$).

It was observed by B\l ocki (\cite{Bl}) that, following the ideas
of Bedford and Taylor ({\cite{BT}, \cite{BT2}), one can relax the
smoothness requirement on $u$ and develop a non linear version of
potential theory for Hessian operators.

The relevant definition is as follows:
\begin{definition} Let $u$ be a subharmonic function on a
domain $\Om\subset\mathbb C^n$. Then $u$ is called $m$-subharmonic (m-sh for short)
if for any collection of $\mathcal C^2$-smooth m-sh functions
$v_1,\ \cdots,\ v_{m-1}$ the inequality
$$dd^cu\we dd^cv_1\we\cdots\we dd^cv_{m-1}\we\beta^{n-m}\geq 0$$
holds in the weak sense of currents.

The set of all $m$-sh functions is denoted by $ \sh _m
(\Om)$.
\end{definition}
Similarly weak $\om$-m-subharmonicity is defined on compact
K\"ahler manifolds (see \cite{DK} and \cite{Chi}).

Using an approximating sequence of functions one can
follow the Bedford and Taylor construction from \cite{BT2} of the
wedge products of currents given by locally bounded $m$-sh
functions. They are defined inductively by
$$dd^cu_1\we\cdots\we dd^cu_p\we\beta^{n-m}:=dd^c (u_1\we\cdots\we dd^cu_p\we\beta^{n-m}).
$$

 It can be shown (see \cite{Bl}) that,
in analogy to the pluripotential setting, these currents are
continuous under monotone or uniform convergence of their
potentials. Moreover the mixed Hessian measures are  also positive
(\cite{Bl}):
\begin{prop}\label{mix}
Let $u_1,\cdots,u_m$ be bounded $m$-subharmonic functions then the measure
$$dd^cu_1\we\cdots\we dd^cu_m\we\beta^{n-m}$$
is nonnegative.
\end{prop}

Here we list some basic facts about $m$-subharmonicity.
 \begin{prop}\label{aaaa}
 Let $\Om\subset\mathbb C^n$\ be a domain. If $\sh_k (\Om)$ denotes the class of $k$-subharmonic
 functions in $\Om$ then
 \begin{enumerate}
 \item $\mathcal{SH}(\Om)=\sh _1 (\Om)\subset \sh _2 (\Om)\subset\cdots
 \subset \sh _n (\Om)=\mathcal{PSH}(\Om)$,
 \item If $u\in \sh _m (\Om)$\ and $\gamma:\mathbb
 R\rightarrow\mathbb R$\ is a $\mathcal C^2$-smooth convex, increasing
 function then $\gamma\circ u\in \sh _m (\Om)$,
 \item Given any convex model domain $P$ the  $P$-average of a
 $m$-sh function $u$ defined by $u_P (z):=\frac{1}{\int _P \beta ^n}
 \int_P u(z+w)\beta^n (w)$
 function is again $m$-sh on the set where it is defined.
 \end{enumerate}
 \end{prop}
Note that, in particular, if $u$ is nonnegative $m$-sh, and $B(r)$
is the ball $\{z\in\mathbb C^n : ||z||<r\}$
 then
$[u^2]_r(z):=\frac{1}{\int _{B(r)} \beta
^n}\int_{B(r)}u^2(z+w)\beta^n (w)$, is still $m$-subharmonic.

We shall use a similar notation for averaging of differential forms

$$[\alpha ]_r(z):=\frac{1}{\int _{B(r)} \beta
^n}\int_{B(r)} \alpha (z+w) \beta ^n (w) .$$

In the sequel we shall exploit the notion of a {\it maximal} $m$-subharmonic function. These were introduced by B\l ocki in \cite{Bl} as analogues of maximal plurisubharmonic functions:
\begin{definition} An $m$-subharmonic function $u$ on a domain $\Om$ is called maximal if for every $m$-subharmonic function $v$ and any compact set $K\subset\Om$ the inequality $u\geq v$ on $\Om\setminus K$ implies $u\geq v$ on $K$.
\end{definition}

The following proposition from \cite{Bl} characterizes maximality for {\it bounded} $m$-subharmonic functions:
\begin{prop} Let $u$ be a bounded $m$-subharmonic function on a domain $\Om$. Then $u$ is maximal if and only if 
$$(dd^cu)^m\we\beta^{n-m}=0$$
as measures.
\end{prop}
The following two theorems, known as comparison principles in
pluripotential theory, follow essentially from the same arguments
as in the case $m=n$:
\begin{theorem}\label{compprin} Let $u,\ v$ be continuous $m$-sh
functions in a bounded domain $\Om\in\mathbb C^n$. Suppose that
$\liminf_{z\rightarrow\partial\Om} (u-v)(z)\geq0$ then

$$\int_{\lbrace u<v\rbrace}(dd^c v)^{m}\we\beta^{n-m}\leq\int_{\lbrace u<v\rbrace}
(dd^c u)^{m}\we\beta^{n-m}.$$
\end{theorem}
\begin{theorem}\label{partialcompprin}  Let $u,\ v,\ w_1,\cdots,w_{m-1}$ be
continuous $m$-sh functions in a bounded domain $\Om\in\mathbb
C^n$. Suppose that
$\liminf_{z\rightarrow\partial\Om}(u-v)(z)\geq0$ then
\begin{align*}&\int_{\lbrace u<v\rbrace} dd^c v \we dd^cw_1\we\cdots\we dd^cw_{m-1}
\we\beta^{n-m}\\
&\leq\int_{\lbrace u<v\rbrace} dd^c u\we dd^cw_1\we\cdots\we
dd^cw_{m-1}\we\beta^{n-m}.
\end{align*}
\end{theorem}

On a compact K\"ahler manifold $(X,\om)$ the solvability of the
Hessian equation (for $1<m<n$)
\begin{equation}\label{hessian}
\begin{cases}u\ is\ \om-m\ {\rm subharmonic};\\ (\om+dd^cu)^{m}\we\om^{n-m}=f\om^n
\\ sup_Xu=0,\end{cases}
\end{equation}
with a strictly positve smooth function $f$ satisfying the
necessary conditon $\int_Xf\om^n=\int_X\om^n$ has attracted
recently a lot of interest.  It should be mentioned that the case $m=1$ corresponds
to the standard Laplace equation, while the solvability in the
case $m=n$ is guearanteed by the Calabi-Yau theorem. In the remaining cases $1<m<n$ the equation was solved  under a
curvature assumption (\cite{Hou}, \cite{J}, \cite{Kok}) or under
additional convexity assumptions for $u$ itself (\cite{G},
\cite{Z}). The continuity method applied in the solution of the
Calabi conjecture by S. T. Yau (\cite{Y}) is already known to
work, up to a point, without any extra assumptions. Let us state
the relevant ingredients:
\begin{enumerate}
\item The openness part for the continuity method works (\cite{Hou});
\item The uniform estimate needed for the closedness part was furnished by Hou
(\cite{Hou}). Alternative approaches were found in \cite{DK} and
\cite{Chi};
\item The $\mathcal C^{2+\alpha}$ a priori estimate for $u$ follows from standard
Evans-Krylov theory as it was observed by Hou (\cite{Hou});
\item The $\mathcal C^2$ a priori estimate holds provided there is an a priori
gradient bound (\cite{HMW}).
\end{enumerate}
In particular the only remaining part to accomplish the program
was the gradient a priori estimate. In fact the Hou-Ma-Wu
\cite{HMW} result yields even more and it will be used in our
argument  below.
\begin{theorem}\label{hmw}
If $u\in \mathcal C^4(X)$ solves the equation (\ref{hessian}) then the following
$\mathcal C^2$ a priori estimate holds
\begin{equation}
sup_X||dd^cu||_{\om}\leq C(sup_X||\nabla u||^2+1),
\end{equation}
where $C$ is a constant dependent on $||f^{1/m}||_{\mathcal C^2(X)},\ m,\ n ,\ ||u||_{\infty}$
and the bound on the bisectional curvature of $(X,\om)$.
\end{theorem}
In particular, as mentioned by the authors in \cite{HMW}, this
type of estimate is amenable to blow-up analysis which reduces the
gradient bound to a certain Liouville type theorem for global
$m$-sh functions on $\mathbb C^n$.
\section{The gradient estimate for the complex Hessian equation - reduction to a Liouville
type theorem}
Although it is pretty standard, we present the details of this
reduction for the sake of completeness.

Suppose that on some compact K\"ahler manifold $(X,\om)$ the gradient estimate fails.
By definiton one then finds a sequence of $\mathcal C^4$ smooth $\om-m$ sh functions
$u_j$ solving the problems
\begin{equation}\label{hessianj}
\begin{cases}(\om+dd^cu_j)^{m}\we\om^{n-m}=f_j\om^n\\ sup_Xu_j=0,\end{cases}
\end{equation}
with some strictly positive functions $f_j$ satisfying
$||f_j ^{1/m}||_{\mathcal C^2(X)}\leq C$ for some uniform constant
$C$ (in particular the uniform norm of $f_j$ is also under control
and  so is the uniform norm of $u_j$ (\cite{Hou}, \cite{DK})), yet
$C_j:=sup_X|\nabla u_j|\rightarrow\infty$.

After passing to a subsequence one may assume that the points
maximizing the functions $||\nabla u_j||$ lie
 in a fixed coordinate chart  and that they converge to the coordinate center.
 We may
further assume that $\om:=i\sum_{k,j=1}^ng_{j\bar{k}}dz_j\we
d\bar{z}_k$  is equal to $dd^cv$ for some local smooth
plurisubharmonic function (the local potential of $\om$).

Changing the coordinates and shrinking the chart if needed one may
further assume that $\om(0)=\beta$ and $\om(z)=\beta +
O(||z||^2).$ Then a subsequence (which we shall still denote by
$u_j$)  has  maximum points of the modulus of the gradient
contained in the chart and converging to the center. Let us denote
these maximum points by $z_k$. Without loss of generality we
assume that the ball of radius $2$ in the introduced coordinates
is contained in the chart.

Consider now the functions $\hat{u}_j(z):=u_j(z_j+\frac{z}{C_j})$.
By construction $\hat{u}_j$ is a $\mathcal C^4$ smooth function on
the ball $\mathbb B(0,C_j):=\{z\in\mathbb C^n : ||z||<C_j\}$ with
$sup_{\mathbb B(0,C_j)}||\nabla \hat{u}_j||=1+O(\frac{1}{C_j})$.
Moreover 
$$sup_{\mathbb
B(0,C_j)}||dd^c\hat{u}_j||\leq\frac{1}{C_j^2}sup_X||dd^cu_j||_{\om}\leq
C$$ 
by Theorem \ref{hmw}. Thus for any $\alpha\in(0,1)$ the family
$\{\hat{u}_j\}_{j\geq k}$ is relatively compact in the $\mathcal
C^{1,\alpha}$ topology on the domain $\mathbb B(0,C_k)$. Then, by
a standard diagonalization procedure and the Arzela-Ascoli theorem
one can extract a limiting function $u\in\mathcal
C^{1,\alpha}(\mathbb C^n)$, such that $||\nabla u(0)||=1$. In
particular $u$ is not constant.

Observe however that in the introduced coordinates (\ref{hessianj}) becomes
$$(dd^c (v+u_j)(z))^m\we (dd^c v(z))^{n-m}=f_j(z)(dd^c v(z))^n.$$
Thus $\hat{u}_j$ satisfies the equation
\begin{align*}(\frac{1}{C_j^2}dd^c v(z_j+z/C_j)&+dd^c\hat{u}_j(z)
) ^{m}\we
(\frac{1}{C_j^2}dd^c v(z_j+z/C_j))^{n-m}\\
&=f_j(z_j+z/C_j) (\frac{1}{C_j^2}dd^c v (z_j+z/C_j))^{n}.
\end{align*}
Since $dd^c v(z)=\beta +
O(||z||^2)$ near zero, we have for $j$ large from  the latter equality  
\begin{equation}\label{approx}[O(1/C_j^2)\beta+dd^c\hat{u}_j]^{m}
\we [(O(1/C_j^2))\beta]^{n-m}= O(1/C_j^{2n})\beta^n.
\end{equation}
More generally for any $1\leq k\leq m$ a similar reasoning coupled
with the $\om-m$ subharmonicity of $u_j$ gives
$$[O(1/C_j^2)\beta+dd^c\hat{u}_j]^{k}\we [(O(1/C_j^2))\beta]^{n-k}\geq 0.$$

These inequalities tell us that the limiting function $u$ is
$m$-sh. Then the equalities  (\ref{approx}) can be read also in
the pluripotential sense and thus one can extract the weak limit
satisfying
\begin{equation}\label{maximal}
(dd^cu)^m\we\beta^{n-m}=0.
\end{equation}
In particular, this means
that $u$ is a maximal $m$-sh function in $\mathbb C^n$. 

Thus we have constructed a uniformly bounded maximal nonconstant
$m$-sh function which is, in addition, globally $\mathcal
C^{1,\alpha}$ smooth with all $\mathcal C^{1,\alpha}$ norms under
control. In particular the gradient is uniformly bounded. Thus a
Liouville type theorem saying that no such function exists would
yield a contradiciton, and hence a (nonexplicit) a priori gradient
estimate.
\section{Liouville theorem for maximal $m$-subharmonic functions}
Any bounded  entire $m$-subharmonic function is in particular
subharmonic. While bounded subharmonic functions do exist in
$\mathbb C^n,\ n\geq 2$, they all must have controlled behavior at
infinity for {\it almost all values} as the
 next theorem shows.
\begin{theorem}[Cartan's lemma, see \cite{HK}]
Let $u$ be a bounded subharmonic function in $\mathbb C^n,\ n\geq
2$. Fix any $q>2n-2$. If $a=sup_{\mathbb C^n}u$, then
$$lim_{||x||\rightarrow\infty, x\in\mathbb C^n\setminus A}u(x)=a,$$
where the set $A$ is contained in, at most, countable collection
of balls $\mathbb B(x_k,r_k)$ such that
$\sum_{k=1}^{\infty}(\frac{r_k}{||x_k||})^q<\infty.$
\end{theorem}

In particular if $u$ is a bounded nonnegative $m$-sh function with
$1=sup_{\mathbb C^n}u$ then for any $z\in\mathbb C^n$
$lim_{r\rightarrow\infty}[u^2]_r (z)=lim_{r\rightarrow\infty}[u]_r
(z)=1$. This fact could have been proven using more elementary
Harnack inequalities, but by referring to Cartan's lemma we make
the argument shorter.

In what follows a function with bounded gradient is a globally Lipshitz function whose Lipshitz constant is under control.

\begin{theorem}[Liouville type theorem for $m$-subharmonic functions]
 Any  $u$ bounded,  m-sh maximal function in $\Bbb C ^n$ with
bounded gradient is constant.
\end{theorem}
\begin{proof}
Arguing by contradiction assume $u$ is nonconstant and $\inf u =0,
\ \sup u =1$ and $||\nabla u|| <c_0 .$

We use induction over $n$. Note that for $n=1$ and $m=n$ the
statement is known.

First suppose that $u$ has the following property: there exists
$\rho >0$, a sequence of mappings $G_k$, each a composition of a
translation and a complex rotation, and a sequence of  balls $ B_k
= B(0 , r_k),\ \ r_k \to\infty$  such that
\begin{equation}\label{1}
[u^2\circ G_k ]_{r_k} (0) + [u\circ G_k ]_{\rho} (0) -2u\circ G_k
(0) \geq 4/3,
\end{equation}
 where we use the notation
$$ [v]_r (z) =  \frac{1}{Vol(B(z,r))}\int_{B(z,r)}v\, \beta ^n
,$$
 $Vol (E)=\int _E \beta ^n$; and
\begin{equation}\label{2}
\lim _{k\to \infty} \int _{B_k}|\frac{\p  u\circ G_k }{\p z_1 }|^2
\beta ^n =0 .
\end{equation}
The functions
$$
u_k =u\circ G_k ,
$$
 are maximal $m$-subharmonic. By the Arzela-Ascoli theorem and a diagonalisation argument one can choose a
subsequence converging locally uniformly to an entire m-sh
function $v$. This function is also maximal by the convergence
theorem \cite{Bl}. From (\ref{2}), and the bound on $||\nabla u||$
it follows that $v$ is constant along the lines with  fixed $z' =
(z_2 , ... , z_n )$.

Indeed, suppose that for some complex  numbers $a,b$ and positive
$c$ we had  $v(a,z_0 ') >v(b,z_0 ') +2c$. Then, by $|\nabla v|<
c_0$, for $\delta = c/(4c_0 )$
\begin{align}
&\inf \{ v(z_1,z'):|z_1 -a|<\delta , |z' -z_0 '|<\delta \} \nonumber \\
>& \sup \{ v(z_1,z'): |z_1 -b|<\delta , |z' -z_0 '|<\delta \} +c .\nonumber
\end{align}
Integrating the directional derivative of $v$ along intervals
parallel to  $[a,b]$ with endpoints $(a',z'), (b',z'), \
b'-a'=a-b;$ in the sets appearing on both sides of the above
formula we have
$$
|\int _{[a',b']} \frac{\partial v}{\partial w} \, dw| >c, \ \
w=\frac{b-a}{|b-a|} .
$$
Then, using the Cauchy-Schwarz inequality
$$\int _{[a',b']}
|\frac{\partial v}{\partial w}|^2 dw >c^2 /(b-a).
$$
From this, via Fubini's theorem, we get for an interval $I$ of
length $\delta$ perpendicular to $[a,b]$ in $z_1$ plane and
$E=[a,b]\times I \times B(z_0 ', \delta )$
$$
\int _E |\frac{\p  v}{\p z_{1} }|^2 \beta ^n > c(n)\delta ^{2n}
c^2/(4|a-b|) =c(n) c^{2(n+1)}/(4^{n+1}c_0 ^{2n}|a-b|)
$$
($c(n)$ - dimensional constant). By the uniform convergence the
analogous estimate would be true for $u_k$, $k$ large enough, in
contradiction to (\ref{2}).

Therefore $v_1 (z')=v_1 (z_2, ..., z_n )=v(z)$ is well defined in
$\Bbb C ^{n-1}$ and inherits the Lipschitz constant from $u$. It
is also m-sh, maximal since for $k\leq m<n$
$$
(dd^c v)^k \we \beta ^{n-k} (z) = (n-m)(\frac{i}{2}dz_1 \we
d\bar{z}_1)\we [(dd^c v_1)^k \we \beta ^{n-k-1}] (z').
$$
Thus, by the induction hypothesis it is constant, but this is
impossible, since: $v\equiv c$  would violate (\ref{1}) as
$$
[u_k ]_{\rho}(0) -u_k (0)\geq 1/3.
$$

Therefore the property stated at the beginning of the proof does
not occur. This means that there is $R>0$,
 such that
if $r>R$ and $z$,  $ \rho>0$ satisfy
$$
[u^2]_r (z) + [u]_{\rho}(z)-2u(z) \geq 4/3 
$$
then for any vector $w$, $|w|=1$
\begin{equation}\label{3}
\frac{1}{Vol(B(z,r))}\int _{B(z,r)}|\frac{\p  u}{\p w }|^2 \beta ^n
\geq c_1 =c_1 (r) .
\end{equation}

Choose origin, $\rho >0$ and $r>R$  so that $u(0)<1/12,
[u]_{\rho}(0)>3/4 $ and $ [u^2]_r (0)
>3/4.$ The choice is possible since for the first inequality one can use $\inf u
=0$, and
 the second and third inequalities for $r, \rho$ large enough follow from the Cartan lemma. Then
$$
0\in U= \{ 2u< [u^2]_r  + [u]_{\rho} - 4/3 \} .
$$
By maximality of $u$ the set $U$ is not bounded. From $dd^c u^2 =
2 (du\we d^c u +udd^c u)$, and (\ref{3}) it follows that
\begin{equation}\label{4}
dd^c [u^2]_{r } \geq 2[udd^c u]_r + c_2 \beta .
\end{equation}
(with $c_2$ depending on $c_1 ,r$ and the average form $[2udd^c
u]_r$ defined in analogy to $[v]_r$ above). Indeed, consider a
simple positive form
$$
\al =\bigwedge_1 ^{n-1} i\al _j \we \bar{\al}_j
$$
normalized by $\alpha\we\be =\be ^n$, with constant coefficients
$(1,0)$ forms $\al _j .$ Vectors dual to $\{\al _j\}$ span a
hyperplane perpendicular to a unit vector $\gamma$. The current
$du\we d^c u$ can be represented in the basis $d\gamma , \al _1 ,
..., \al _{n-1}$ as
$$
du\we d^c u =|\frac{\partial u}{\partial \gamma}|^2
(\frac{i}{2}d\gamma \we d\bar{\gamma }) +\Theta ,
$$
with $\Theta$ containing differentials in which some $\al _j $'s
do appear. Then (\ref{3}) leads to
$$
[du\we d^c u]_r\we \alpha = [|\frac{\partial u}{\partial
\gamma}|^2 \frac{i}{2}d\gamma  \we d\bar{\gamma }]_r \we\al \geq
c_2 (\frac{i}{2}d\gamma  \we d\bar{\gamma }) \we\al =c_2\be ^n .
$$
This proves (\ref{4}), and hence, by Proposition \ref{mix} we have
\begin{equation}\label{5}
dd^c [u^2]_{r }\wedge (dd^c u)^{m-1}\wedge \beta ^{n-m} \geq c_2
(dd^c u)^{m-1}\wedge \beta ^{n-m+1} ,
\end{equation}
on the set $U$.
Take $g(z) =-\epsilon ||z||^2$ with positive $\ep$ so small that
$c_2 \beta
>- 2dd^c g$ (and therefore $[u^2]_{r }+g$ is m-sh on $U$ (see
(\ref{4})))
 and that
$$
U_g = \{ 2u< [u^2]_r (z) + u_{\rho} +g -4/3 \}
$$
is nonempty. The set $U_g$ is bounded since $\lim _{|z|\to\infty
}g(z)=-\infty$ and it is contained in $U$. Then using
(\ref{4}) and the comparison principle 
\begin{align}
&\frac{c_2}{2}\int _{U_g} (dd^c u)^{m-1}\wedge \beta ^{n-m+1}
\nonumber \\ \leq &\int _{U_g} dd^c ([u^2]_{r }+[u]_{\rho}
+g)\wedge(dd^c u)^{m-1}\wedge \beta ^{n-m} \leq 2^m\int _{U_g} (dd^c
u)^{m}\wedge \beta ^{n-m} ,\nonumber
\end{align}
 and having $0$ on the left, we obtain
$$
(dd^c u)^{m-1}\wedge \beta ^{n-m+1} =0$$ on $U_g$.  Apply  the
comparison principle again for the same functions and the
background form $(dd^c u)^{m-2}\we \beta ^{n-m+2}$ in place of
 $(dd^c u)^{m-1}\wedge \beta ^{n-m+1}$:
\begin{align}
&\frac{c_2}{2}\int _{U_g} (dd^c u)^{m-2}\wedge \beta ^{n-m+2} \nonumber \\
\leq &\int _{U_g} dd^c ([u^2]_{r }+[u]_{\rho} +g)\wedge(dd^c
u)^{m-2}\wedge \beta ^{n-m+1}\\ \leq & \, 2^{m-1}\int _{U_g} (dd^c
u)^{m-1}\wedge \beta ^{n-m+1} \nonumber
\end{align}
Therefore $$(dd^c u)^{m-2}\wedge \beta ^{n-m+2}  =0$$ on $U_g .$
Repeating the procedure, after $m-2$ steps we get $\int
_{U_g}\beta ^n =0$, a contradiction, since $U_g$ is nonempty and
open.

\end{proof}

{\bf Acknowledgment} The authors were partially supported by  NCN grant 2011/01/B/ST1/00879.

{\noindent Rutgers University, Newark, NJ 07102, USA;\\
Faculty of Mathematics and Computer Science,
Jagiellonian University 30-348 Krak\'ow, \L ojasiewicza 6,
Poland;\\ e-mail: {\tt slawomir.dinew@im.uj.edu.pl}}\\ \\

{\noindent Faculty of Mathematics and Computer Science,
Jagiellonian University 30-348 Krak\'ow, \L ojasiewicza 6,
Poland;\\ e-mail: {\tt slawomir.kolodziej@im.uj.edu.pl}}\\ \\

\end{document}